\documentclass[12pt]{amsart}
\usepackage{amsfonts,amssymb,amscd,amsmath,enumerate,verbatim,color}
\usepackage[latin1]{inputenc}
\usepackage{amscd}
\usepackage{latexsym}

\usepackage[dvipdfmx]{graphicx}
\usepackage{mathptmx}

\def\NZQ{\mathbb}               

\def\ZZ{{\NZQ Z}}
\def\RR{{\NZQ R}}

\def\ab{{\mathbf a}}

\def\eb{{\mathbf e}}
\def\tb{{\mathbf t}}

\def\wb{{\mathbf w}}
\def\xb{{\mathbf x}}
\def\yb{{\mathbf y}}

\def\opn#1#2{\def#1{\operatorname{#2}}} 
\opn\gr{gr}

\def\Ac{{\mathcal A}}

\def\Jc{{\mathcal J}}
\def\Gc{{\mathcal G}}
\def\Fc{{\mathcal F}}
\def\Oc{{\mathcal O}}
\def\Pc{{\mathcal P}}
\def\Qc{{\mathcal Q}}

\def\Cc{{\mathcal C}}

\def\comax{{\textnormal{comax}}}
\def\comin{{\textnormal{comin}}}
\newtheorem{Theorem}{Theorem}[section]
\newtheorem{Lemma}[Theorem]{Lemma}
\newtheorem{Corollary}[Theorem]{Corollary}
\newtheorem{Proposition}[Theorem]{Proposition}

\theoremstyle{definition}

\newtheorem{Example}[Theorem]{Example}

\let\epsilon\varepsilon
\let\phi=\varphi
\let\kappa=\varkappa
\opn\dis{dis}
\opn\height{height}
\opn\dist{dist}
\def\pnt{{\raise0.5mm\hbox{\large\bf.}}}

\opn\Lex{Lex}
\opn\conv{conv}

\begin{document}

\title{Enriched order polytopes and Enriched Hibi rings}
\author{Hidefumi Ohsugi and Akiyoshi Tsuchiya}
\address{Hidefumi Ohsugi,
	Department of Mathematical Sciences,
	School of Science and Technology,
	Kwansei Gakuin University,
	Sanda, Hyogo 669-1337, Japan} 
\email{ohsugi@kwansei.ac.jp}

\address{Akiyoshi Tsuchiya,
Graduate School of Mathematical Sciences,
University of Tokyo,
Komaba, Meguro-ku, Tokyo 153-8914, Japan} 
\email{akiyoshi@ms.u-tokyo.ac.jp}


\begin{abstract}
Stanley introduced two classes of lattice polytopes associated to posets,
which are called the order polytope ${\mathcal O}_P$ and the chain polytope ${\mathcal C}_P$ of a poset $P$.
It is known that, given a poset $P$, the Ehrhart polynomials of ${\mathcal O}_P$ and ${\mathcal C}_P$
are equal to the order polynomial of $P$ that counts the $P$-partitions. 
In this paper, we introduce the enriched order polytope of a poset $P$
and show that it is a reflexive polytope whose Ehrhart polynomial is equal to that of 
the enriched chain polytope of $P$ and the left enriched order polynomial of $P$
 that counts the left enriched $P$-partitions,
by using the theory of Gr\"{o}bner bases.
The toric rings of enriched order polytopes are called enriched Hibi rings.
It turns out that enriched Hibi rings are normal, Gorenstein, and Koszul.
The above result implies the existence of a bijection between
the lattice points in the dilations of $\Oc^{(e)}_P$ and  
$\Cc^{(e)}_P$.
Towards such a bijection, we give the facet representations of enriched order and chain polytopes.
\end{abstract}

\maketitle

\section{Introduction}
A {\em lattice polytope} in $\RR^n$ is a convex polytope all of whose vertices are in $\ZZ^n$.
In \cite{twoposetpolytopes}, Stanley introduced a class of lattice polytopes associated to finite partially ordered sets ({\em poset} for short).
Let $(P, <_P)$ be a finite poset on $[n]:=\{1,\ldots,n\}$.
The {\em order polytope} $\Oc_P$ of $P$ is the convex polytope consisting of the set of points $(x_1,\ldots,x_n) \in \RR^n$ such that
\begin{enumerate}
	\item $0 \leq x_i \leq 1$ for $1 \leq i \leq n$,
	\item $x_i \leq x_j$ if $i <_P j$.
\end{enumerate}
Then $\Oc_P$ is a lattice polytope of dimension $n$.
In fact, each vertex of $\Oc_P$ corresponds to a filter of $P$.
Here, a subset $F$ of $P$ is called a {\em filter} of $P$ if $i \in F$ and $j \in P$ together with $i <_P j$ guarantee $j \in F$.
For a subset $X \subset [n]$, we define the $(0,1)$-vector $\eb_X:=\sum_{i \in X}\eb_i$, where $\eb_1,\ldots,\eb_n$ are the canonical unit coordinate vectors of $\RR^n$.
Then one has $\Oc_P \cap \ZZ^n=\{\eb_F : F \in \Fc(P)\}$, where $\Fc(P)$ is the set of filters of $P$. 
Moreover, there is a close interplay between the combinatorial structure of $P$ and the geometric structure of $\Oc_P$.
Assume that $P$ is {\em naturally labeled}, i.e., $i < j$ if $i <_P j$.
Let $\ZZ_{\ge 0}$ be the set of nonnegative integers.
A map $f : P \to \ZZ_{\ge 0}$ is called a {\em $P$-partition} if for all $x,y \in P$ with $x <_P y$, $f$ satisfies $f(x) \leq f(y)$.
We identify a $P$-partition $f$ with a lattice point $(f(1),\ldots,f(n)) \in \ZZ^n$.
Since every $P$-partition $f : P \to \ZZ_{\ge 0}$ with $f(i) \leq 1$ is a filter of $P$, the set of $P$-partitions $f : P \to \ZZ_{\ge 0}$ with $f(i) \leq 1$ coincides with $\Oc_P \cap \ZZ^n$. 
Moreover, the set of $P$-partitions $f : P \to \ZZ_{\ge 0}$ with $f(i) \le m$ coincides with $m\Oc_P \cap \ZZ^n$
for $0 < m \in \ZZ$.
Here, for a convex polytope $\Pc \subset \RR^n$, $m\Pc:=\{m\xb : \xb \in \Pc\}$ is the $m$-th dilated polytope.

In the present paper, we define a new class of lattice polytopes associated to finite posets from a viewpoint of the theory of enriched $P$-partitions.
For a filter $F$ of $P$, we set $F_{\text{min}}:=\min(F)$ and $F_{{\text{comin}}}:=F \setminus F_{\min}$, where $\min(F)$ is the set of minimal elements of $F$.
For a subset $X=\{i_1,\ldots,i_r\} \subset [n]$ and a vector $\varepsilon=(\varepsilon_1,\ldots,\varepsilon_r) \in \{-1,1\}^{r}$, we define the $(-1,0,1)$-vector $\eb_X^{\varepsilon}:=\sum_{j=1}^{r} \varepsilon_j\eb_{i_j}$.
The \textit{enriched order polytope} $\Oc^{(e)}_P \subset \RR^n$ of a finite (not necessarily naturally labeled) poset $P$ on $[n]$ is the lattice polytope of dimension $n$ which is the convex hull of
\begin{equation}
\label{latticeset}
	\{
\eb_{F_{\min}}^{\varepsilon} + \eb_{F_{\comin}}
:
F \in \Fc(P), \varepsilon \in \{-1,1\}^{|F_{\min}|}
\}.
\end{equation}
Then $\Oc^{(e)}_P \cap \ZZ^n$ coincides with the set (\ref{latticeset}) above (Lemma \ref{lem:latticepoint}).
Now, we discuss a relation between $\Oc^{(e)}_P$ and the theory of enriched $P$-partitions.
Again, we assume that $P$ is naturally labeled. 
A map $f: P \rightarrow \ZZ \setminus \{0\}$ is called an 
{\em enriched $P$-partition} (\cite{StembridgeEnriched})
if, for all $x, y \in P$ with $x <_P y$, $f$ satisfies
\begin{itemize}
	\item
	$|f(x)| \le |f(y)|$;
	\item
	$|f(x)| = |f(y)|  \ \Rightarrow \ f(y) > 0$.
\end{itemize}
On the other hand,
Petersen \cite{Petersen} introduced slightly different notion ``left enriched $P$-partitions'' as follows.
A map $f: P \rightarrow \ZZ$ is called a {\em left enriched $P$-partition} if, for all $x, y \in P$ with $x <_P y$, $f$ satisfies
the following conditions:
\begin{itemize}
	\item[(i)]
	$|f(x)| \le |f(y)|$;
	\item[(ii)]
	$|f(x)| = |f(y)| \ \Rightarrow \ f(y) \ge 0$.
\end{itemize}
Then the set of left enriched $P$-partitions $f : P \to \ZZ$ with $|f(i)| \le 1$ coincides with $\Oc^{(e)}_P \cap \ZZ^n$.
Contrary to the case of order polytopes, the set of left enriched $P$-partitions $f : P \to \ZZ$
with $|f(i)| \le m$ does {\em not} always coincide with the set of lattice points $m\Oc^{(e)}_P \cap \ZZ^n$ for $m >1$ 
(Example~\ref{2chain}).
However, we will show that the number $\Omega_P^{(\ell)}(m)$ of left enriched $P$-partitions 
$f : P \to \ZZ$ with $|f(i)| \le m$ is equal to $|m\Oc^{(e)}_P \cap \ZZ^n|$.
Namely, 

\begin{Theorem}\label{thm:ehrhart}
For a naturally labeled finite poset $P$ on $[n]$,
let 
$$L_{\Oc^{(e)}_P}(m) = | m \Oc^{(e)}_P \cap \ZZ^n|$$
be the Ehrhart polynomial 
of $\Oc^{(e)}_P$, and let $\Omega_P^{(\ell)}(m)$ be the left enriched order polynomial 
 of $P$.
	Then one has 
$$
	L_{\Oc^{(e)}_P}(m) = L_{\Oc^{(e)}_{\overline{P}}}(m) = \Omega_P^{(\ell)} (m),
$$
where $\overline{P}$ is the dual poset of $P$.
	\end{Theorem}

In this paper, in order to show Theorem~\ref{thm:ehrhart}, we investigate the toric ring of the enriched order polytope $\Oc^{(e)}_{\overline{P}}$.
In \cite{Hibi1987}, Hibi studied the toric ring of the order polytope $\Oc_{\overline{P}}$.
The toric ideal $I_{\Oc_{\overline{P}}}$ possesses a squarefree quadratic Gr\"{o}bner basis, that is a Gr\"{o}bner basis consisting of binomials whose initial monomials are squarefree and of degree $2$.
This implies that the toric ring $K[\Oc_{\overline{P}}]$ is a normal Cohen-Macaulay domain and Koszul.
In particular, $\Oc_{\overline{P}}$ possesses a flag regular unimodular triangulation.
The toric ring $K[\Oc_{\overline{P}}]$ is called the {\em Hibi ring} of $P$.
See \cite[Chapter~6]{binomialideals}.
We call the toric ring $K[\Oc^{(e)}_{\overline{P}}]$ the {\em enriched Hibi ring} of $P$.

\begin{Theorem}
\label{thm:flagetc}
	Let $P$ be a finite poset on $[n]$.
	Then $\Oc^{(e)}_P$ is a reflexive polytope with a flag regular unimodular triangulation.
	Moreover, the toric ring $K[\Oc^{(e)}_P]$ is a normal, Gorenstein, and Koszul.
\end{Theorem}

First, in Section~2, we introduce known results on two poset polytopes introduced by Stanley \cite{twoposetpolytopes}, that is,
the order polytope $\Oc_P$ and the chain polytope $\Cc_P$ of a poset $P$.
A squarefree quadratic Gr\"{o}bner basis of the toric ideal of each of $\Oc_P$, $\Cc_P$ and its applications
will be extended to ``enriched case'' in the following sections.
In Section~3, we study the notion of enriched chain polytopes $\Cc_P^{(e)}$ (\cite{OTenriched})
because we need to compare the toric ideals of enriched order polytopes and that of 
enriched chain polytopes in order to prove Theorem~\ref{thm:ehrhart}.
In Section~4, we discuss fundamental properties of enriched order polytopes.
In Section~5, we study the toric ideals of enriched order polytopes and their applications.
%
By proving that the toric ideal of $\Oc^{(e)}_{\overline{P}}$ possesses a squarefree quadratic Gr\"{o}bner basis consisting of binomials whose initial monomials do not contain the variable corresponding to the origin, 
we show Theorem~\ref{thm:flagetc} (Corollary~\ref{maincorollary}).
Moreover, by comparing the initial ideals of toric ideals of two enriched poset polytopes  (Theorem~\ref{lasttheorem}), we will complete the proof of Theorem \ref{thm:ehrhart}.
Note that Theorem~\ref{thm:ehrhart} implies the existence of a bijection between $m \Oc^{(e)}_P \cap \ZZ^n$ and  
$m \Cc^{(e)}_P \cap \ZZ^n$.
In Section~6, towards such a bijection, 
we consider an elementary geometric property,
 the facet representations of enriched order and chain polytopes
(Proposition~\ref{facetsofC}, Theorem~\ref{facetsofO}).
The number of facets is discussed in 
Corollary~\ref{numberoffacets},
and Proposition \ref{upperbound}.
Finally, we show that
$\Oc^{(e)}_P$ is rarely unimodularly equivalent to $\Cc^{(e)}_P$ (Proposition~\ref{rare}).

\subsection*{Acknowledgment}
The authors are grateful to an anonymous referee for his useful comments.
In particular, the last section was added following his advice. 
The authors were partially supported by JSPS KAKENHI 18H01134 and 16J01549.

\section{Two poset polytopes}
In this section, we review properties of order polytopes and  chain polytopes. 
Let ${(P,<_P)}$ be a finite poset on $[n]$.
Recall that the order polytope $\Oc_P \subset \RR^n$ is the convex hull of 
\[
\{\eb_F : F \in \Fc(P) \}.
\]
In \cite{twoposetpolytopes}, Stanley introduced another lattice polytope associated to $P$ as well as the order polytope $\Oc_P$.
An {\em antichain} of $P$ is a subset of $P$ consisting of pairwise incomparable elements of $P$.
Let $\Ac(P)$ denote the set of antichains of $P$.
Note that the empty set $\emptyset$ is an antichain of $P$.
The {\em chain polytope} $\Cc_P \subset \RR^n$ of $P$
is the convex hull of 
\[
\{
\eb_A :
A \in \Ac(P)
\}.
\]
Then $\Cc_P$ is a lattice polytope of dimension $n$.
The order polytope $\Oc_P$ and the chain polytope $\Cc_P$ have similar properties. 

First, we study the Ehrhart polynomials of $\Oc_P$ and $\Cc_P$.
Let $\Pc \subset \RR^n$ be a general lattice polytope of dimension $n$. 
Given a positive integer $m$, we define
\[
L_{\Pc}(m) = |m \Pc \cap \ZZ^n|.
\]
The study on $L_{\Pc}(m)$ originated in Ehrhart \cite{Ehrhart} who proved that $L_{\Pc}(m)$ is a polynomial in $m$ of degree $n$ with the constant term $1$. Moreover, the leading coefficient of $L_{\Pc}(m)$ coincides with the usual Euclidean volume of $\Pc$. 
We say that $L_{\Pc}(m)$ is the \textit{Ehrhart polynomial} of $\Pc$.
An Ehrhart polynomial often coincides with a counting function of a combinatorial object.
A map $f : P \to \ZZ_{\ge 0}$ is called an \textit{order preserving map} if for all  $x,y \in P$ with $x <_P y$, $f$ satisfies $f(x) \leq f(y)$.
Let $\Omega_P(m)$ denote the number of order preserving maps $f : P \to \ZZ_{\ge 0}$ with $f(i) \le m$.
Then $\Omega_P(m)$ is a polynomial in $m$ of degree $n$ and called the \textit{order polynomial} of $P$.
Stanley showed a relation between the Ehrhart polynomials of $\Oc_P$ and $\Cc_P$ and the order polynomial $\Omega_P(m)$.
In fact,
\begin{Proposition}[{\cite[Theorem 4.1]{twoposetpolytopes}}]
\label{twoEhrhartOmega}
	Let $P$ be a finite poset on $[n]$.
	Then one has
	\[
	L_{\Oc_P}(m)=L_{\Cc_P}(m)=\Omega_P(m+1).
	\]
\end{Proposition}

On the other hand, $\Oc_P$ and $\Cc_P$ are not always unimodularly equivalent.
Here, two lattice polytopes $\Pc,\Qc \subset \RR^n$ of dimension $n$ are \textit{unimodularly equivalent} if there exist a unimodular matrix $U \in \ZZ^{n \times n}$ and a lattice point $\wb \in \ZZ^n$ such that $\Qc = f_U(\Pc)+\wb$, where $f_U$ is the linear transformation in $\RR^n$ defined by $U$, i.e., $f_U(\xb)=\xb U$ for all $\xb \in \RR^n$.
In \cite{uniequiv}, Hibi and Li characterized when $\Oc_P$ and $\Cc_P$ are unimodularly equivalent.
In fact,
\begin{Proposition}[{\cite[Corollary 2.3]{uniequiv}}]
\label{Xposet}
	Let $P$ be a finite poset on $[n]$.
Then the following conditions are equivalent{\rm :}
	\begin{enumerate}
		\item[{\rm (i)}] 
The order polytope $\Oc_P$ and the chain polytope $\Cc_P$ are unimodularly equivalent{\rm ;}
		\item[{\rm (ii)}]
The number of the facets of $\Oc_P$ is equal to that of $\Cc_P${\rm ;}
 		\item[{\rm (iii)}] 
The following poset is not a subposet of $P$.	
		\newline
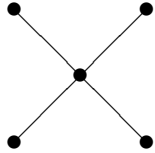
		\begin{center}
	\begin{picture}(200,50)(10,30)

	\put(75,90){\circle*{5}}
	\put(125,90){\circle*{5}}
	\put(75,40){\circle*{5}}
	\put(125,40){\circle*{5}}
	\put(100,65){\circle*{5}}
	\put(75,90){\line(1,-1){50}}
	\put(125,90){\line(-1,-1){50}}
	
	\end{picture}		\end{center}
	\end{enumerate}

\end{Proposition}

Next, we review the toric ideals of order polytopes and chain polytopes. 
First, we recall basic materials and notation on toric ideals.
Let $K[{\bf t}^{\pm1}, s] = K[t_{1}^{\pm1}, \ldots, t_{n}^{\pm1}, s]$
be the Laurent polynomial ring in $n+1$ variables over a field $K$. 
If $\ab = (a_{1}, \ldots, a_{n}) \in \ZZ^{n}$, then
${\bf t}^{\ab}s$ is the Laurent monomial
$t_{1}^{a_{1}} \cdots t_{n}^{a_{n}}s \in K[{\bf t}^{\pm1}, s]$. 
Let  $\Pc \subset \RR^{n}$ be a lattice polytope and $\Pc \cap \ZZ^n=\{\ab_1,\ldots,\ab_d\}$.
Then, the \textit{toric ring}
of $\Pc$ is the subalgebra $K[\Pc]$ of $K[{\bf t}^{\pm1}, s]$
generated by
$\{\tb^{\ab_1}s ,\ldots,\tb^{\ab_d}s \}$ over $K$.
We regard $K[\Pc]$ as a homogeneous algebra by setting each $\text{deg } \tb^{\ab_i}s=1$.
Let $K[\xb]=K[x_1,\ldots, x_d]$ denote the polynomial ring in $d$ variables over $K$ with each $\deg(x_i)=1$.
The \textit{toric ideal} $I_{\Pc}$ of $\Pc$ is the kernel of the surjective homomorphism $\pi : K[\xb] \rightarrow K[\Pc]$
defined by $\pi(x_i)=\tb^{\ab_i}s$ for $1 \leq i \leq d$.
It is known that $I_{\Pc}$ is generated by homogeneous binomials.
See, e.g., \cite{binomialideals, sturmfels1996}.

Now, we study the toric ideals of $\Oc_{\overline{P}}$ and $\Cc_{\overline{P}}$.
Remark that $\Oc_P$ and $\Oc_{\overline{P}}$ are unimodularly equivalent and $\Cc_P=\Cc_{\overline{P}}$.
A subset $I$ of $P$ is called a {\em poset ideal} of $P$ if $i \in I$ and $j \in P$ together with $i >_P j$ guarantee $j \in I$.
Let $\Jc(P)$ denote the set of poset ideals of $P$, ordered by inclusion.
If $I \in \Jc(P)$ and $J \in \Jc(P)$ are incomparable in $\Jc(P)$, then we write $I \nsim J$.
Then the order polytope $\Oc_{\overline{P}}$ is the convex hull of
\[\{
\eb_I : I \in \Jc(P)
\}.\]
Let $R[\Oc]$ denote the polynomial ring over $K$ in variables $x_I$, where $I \in \Jc(P)$.
In particular, the origin corresponds to the variable $x_{\emptyset}$.
Then the toric ideal $I_{\Oc_{\overline{P}}}$ is the kernel of the ring homomorphism $\pi_{\Oc} : R[\Oc] \to K[t_1,\dots, t_n, s]$ defined by $\pi_{\Oc}(x_I)=s \prod_{i \in I}t_i$.
Let $<_{\Oc}$ be a reverse lexicographic order on $R[\Oc]$ such that $x_I <_{P} x_J$ if $I \subsetneq J$.
In \cite{Hibi1987}, Hibi essentially proved that $I_{\Oc_{\overline{P}}}$ possesses a squarefree quadratic Gr\"{o}bner basis.
In fact,
\begin{Proposition}[{\cite{Hibi1987}}]
	\label{thm:order}
	Work with the same notation as above.
Then
$$
\Gc_{\Oc} = \{ 	x_I x_J - x_{I \cup J}\ x_{I \cap J} : I, J \in \Jc(P) , I \nsim J   \}
$$
is a Gr\"{o}bner basis of $I_{\Oc_{\overline{P}}}$ 
with respect to a reverse lexicographic order $<_{\Oc}$. 
Moreover, $R[\Oc]/I_{\Oc_{\overline{P}}}$ is a normal Cohen-Macaulay domain and Koszul.
\end{Proposition} 
Recently, the toric ring $K[\Oc_P] \cong K[\Oc_{\overline{P}}] \cong R[\Oc]/I_{\Oc_{\overline{P}}}$ is called the \textit{Hibi ring} of $P$ and studied by many authors from several viewpoints.
One can find some of them in \cite[Note of Chapter~6]{binomialideals}.

For a poset ideal $I$ of $P$, we denote $\max(I)$ the set of maximal elements of $I$.
Then $\max(I)$ is an antichain of $P$ and every antichain of $P$ is the set of maximal elements of a poset ideal.
On the other hand, for an antichain $A$ of $P$, the poset ideal of $P$ generated by $A$ is
the smallest poset ideal of $P$ which contains $A$. Every poset ideal of $P$ can be obtained by this way.
Hence $\Jc(P)$ and $\Ac(P)$ have a one-to-one correspondence.
Let $R[\Cc]$ denote the polynomial ring over $K$ in variables $x_{\max(I)}$, where $I \in \Jc(P)$.
Then the toric ideal $I_{\Cc_{\overline{P}}}$ is the kernel of the ring homomorphism $\pi_{\Cc} : R[\Cc] \to K[t_1,\dots, t_n, s]$ defined by $\pi_{\Cc}(x_{\max(I)})=s \prod_{i \in \max(I)} t_i$.
Let $<_{\Cc}$ be a reverse lexicographic order on $R[\Cc]$ such that $x_{\max(I)} <_{\Cc} x_{\max(J)}$ if $I \subsetneq J$.
Given poset ideals $I, J \in \Jc(P)$, 
let $I * J$ denote the poset ideal of $P$
generated by $\max(I \cap J) \cap (\max(I) \cup \max(J))$.
Note that $I*J \subset I \cap J$.
The following lemma is fundamental and important.

\begin{Lemma}
\label{fourconditions}
Let $P$ be a finite poset and $I, J \in \Jc(P)$.
For $p \in \max(I) \setminus \max(J)$,
 the following conditions are equivalent{\rm :}
	\begin{enumerate}
		\item[{\rm (i)}] $p \in J${\rm ;}
		\item[{\rm (ii)}]  $p \in \max(I \cap J)${\rm ;}
		\item[{\rm (iii)}]  $p \in \max(I * J)${\rm ;}
		\item[{\rm (iv)}]  $p \notin \max(I \cup J)$.
	\end{enumerate}
\end{Lemma}

\begin{proof}
First, (ii) $\Rightarrow$ (i) is trivial.
Suppose $p \in J$.
Since $p$ does not belong to $\max (J)$,
$p$ is not a maximal element in $J$.
Hence we have $p \notin \max(I \cup J)$.
Thus (i) $\Rightarrow$ (iv) holds.
Suppose $p \notin \max(I \cup J)$.
Then there exists an element $q \in I \cup J$ 
such that $p <_P q$.
If $q$ belongs to $I$, then $p \notin \max(I)$, a contradiction.
Thus $q \in J$, and hence $p \in I \cap J$.
If $p \notin \max(I \cap J)$ holds,
then there exists an element $q' \in I \cap J$ 
such that $p <_P q'$.
This contradicts to the hypothesis $p \in \max(I)$.
Thus (iv) $\Rightarrow$ (ii) holds.
Finally, we have (ii) $\Leftrightarrow$ (iii) by
$\max (I*J) = \max(I \cap J) \cap (\max(I) \cup \max(J))$.
\end{proof}
In \cite{HibiLi}, Hibi and Li essentially proved that $I_{\Cc_{\overline{P}}}$ possesses a squarefree quadratic Gr\"{o}bner basis.
In fact,
\begin{Proposition}[{\cite{HibiLi}}]
	\label{thm:chain}
	Work with the same notation as above. 
Then
$$
\Gc_{\Cc} = \{ x_{\max(I)} x_{\max(J)} - x_{\max(I \cup J)} x_{\max(I * J)} : 
I, J \in \Jc(P), I \nsim J \}
$$
is a Gr\"{o}bner basis of $I_{\Cc_{\overline{P}}}$ with respect to a reverse lexicographic order $<_{\Cc}$. Moreover,  $R[\Cc]/I_{\Cc_{\overline{P}}}$ is a normal Cohen-Macaulay domain and Koszul.
\end{Proposition} 

From Propositions \ref{thm:order} and \ref{thm:chain} we can prove the following.

\begin{Proposition}
\label{iso}
	Work with the same notation as above.
	Then one has
	\[
	\dfrac{R[\Oc]}{{\rm in}_{<_{\Oc}}(I_{\Oc_{\overline{P}}}) } \cong \dfrac{R[\Cc]}{{\rm in}_{<_{\Cc}}(I_{\Cc_{\overline{P}}})}.
	\]
	Furthermore, we obtain
$
	L_{\Oc_{P}}(m)=L_{\Oc_{\overline{P}}}(m)=L_{\Cc_{\overline{P}}}(m)=L_{\Cc_{P}}(m).
$
\end{Proposition}

\begin{proof}
From Propositions \ref{thm:order} and \ref{thm:chain}, we have
\begin{eqnarray*} 
	{\rm in}_{<_{\Oc}}(I_{\Oc_{\overline{P}}}) &=& (x_I x_J : I, J \in \Jc(P), I \nsim J),\\
	{\rm in}_{<_{\Cc}}(I_{\Cc_{\overline{P}}}) &=& (x_{\max(I)} x_{\max(J)} : I, J \in \Jc(P), I \nsim J).
\end{eqnarray*}
	Hence it follows that the map $x_{I} \mapsto x_{\max(I)}$ induces an isomorphism from $R[\Oc]/{\rm in}_{<_{\Oc}}(I_{\Oc_{\overline{P}}})$ to $R[\Cc]/{\rm in}_{<_{\Cc}}(I_{\Cc_{\overline{P}}})$.
	Therefore, the first claim follows.
	
Since both ${\rm in}_{<_{\Oc}}(I_{\Oc_{\overline{P}}}) $
and ${\rm in}_{<_{\Cc}}(I_{\Cc_{\overline{P}}})$ are squarefree, 
both $\Oc_{\overline{P}}$ and $\Cc_{\overline{P}}$ possesses a unimodular triangulation,
and hence 
the Ehrhart polynomial coincides with the Hilbert polynomial of its toric ring
for each of $\Oc_{\overline{P}}$ and $\Cc_{\overline{P}}$ (see \cite[Section 4.2]{binomialideals} or \cite[Chapters 8 and 13]{sturmfels1996}).
	Moreover, for an ideal $I$ of $K[\xb]$ and a monomial order $<$ on $K[\xb]$, the Hilbert polynomial of $K[\xb]/I$ is equal to that of $K[\xb]/{\rm in}_{<}(I)$.
	Therefore, the second claim follows.
\end{proof}

\section{Enriched chain polytopes}

In this section, we recall the definition and properties of enriched chain polytopes given in \cite{OTenriched}.
Let $(P,<_P)$ be a finite poset on $[n]$.
The {\em enriched chain polytope} $\Cc^{(e)}_P \subset \RR^n$ of $P$ is the convex hull of 
\[
\{
\eb_A^{\epsilon} 
:
A \in \Ac(P), \epsilon \in \{-1,1\}^{|A|}
\}.
\]
Then $\Cc^{(e)}_{P}$ is a lattice polytope of dimension $n$.
It is easy to see that $\Cc^{(e)}_P$ is centrally symmetric
(i.e., for any facet $\Fc$ of $\Cc^{(e)}_P$, $-\Fc$ is also a facet of $\Cc^{(e)}_P$), and the origin of $\RR^n$ is the unique interior lattice point of $\Cc^{(e)}_P$.
Remark that $\Cc^{(e)}_P=\Cc^{(e)}_{\overline{P}}$.

A lattice polytope $\Pc \subset \RR^n$ of dimension $n$ is called \textit{reflexive} if the origin of $\RR^n$ is a unique lattice point belonging to the interior of $\Pc$ and its dual polytope
\[\Pc^\vee:=\{\yb \in \RR^n  :  \langle \xb,\yb \rangle \leq 1 \ \text{for all}\  \xb \in \Pc \}\]
is also a lattice polytope, where $\langle \xb,\yb \rangle$ is the usual inner product of $\RR^n$.
It is known that reflexive polytopes correspond to Gorenstein toric Fano varieties, and they are related to
mirror symmetry (see, e.g., \cite{mirror,Cox}).
In each dimension there exist only finitely many reflexive polytopes 
up to unimodular equivalence (\cite{Lag})
and all of them are known up to dimension $4$ (\cite{Kre}).
Recently, several classes of reflexive polytopes were constructed by an algebraic technique on Gr\"{o}bner bases (c.f., \cite{HTomega,HTperfect,harmony}).
The algebraic technique is based on
the following lemma that follows from the argument in \cite[Proof of Lemma~1.1]{HMOS}.

\begin{Lemma}
	\label{genten_yowaku}
	Let $\Pc \subset \RR^n$ be a lattice polytope of dimension $n$ such that the origin of $\RR^n$ is contained 
	in its interior.
	Suppose that any lattice point in $\ZZ^n$ is a linear integer combination of the lattice points in $\Pc$.
	If 
	there exists a monomial order such that
	the initial ideal of $I_\Pc$ is generated by squarefree monomials which do not contain
	the variable corresponding to the origin,
	then $\Pc$ is reflexive and has a regular unimodular triangulation.
	Moreover, $K[\Pc]$ is a normal Gorenstein domain.
\end{Lemma}

In order to use Lemma \ref{genten_yowaku} for enriched chain polytopes $\Cc^{(e)}_{\overline{P}}$, we study
the toric ideal of $\Cc^{(e)}_{\overline{P}}$. 
Let $R[\Cc^{(e)}]$ denote the polynomial ring over $K$ in variables $x_{A}^{\epsilon}$, where $A \in \Ac(P)$
and $\epsilon=(\epsilon_1,\ldots,\epsilon_n) \in \{-1,0,1\}^{n}$ with
\[
|\epsilon_i|=
\begin{cases}
1  & (i \in  A);\\
0 & (i \notin A).
\end{cases}
\]
Then the toric ideal $I_{\Cc^{(e)}_{\overline{P}}}$ is the kernel of a ring homomorphism 
$\pi_{\Cc^{(e)}}: R[\Cc^{(e)}] \rightarrow K[\tb^{\pm1}, s]$
defined by
$\pi_{\Cc^{(e)}}(x_{A}^{\epsilon}) = t_{1}^{\varepsilon_1}
\dots t_{n}^{\varepsilon_n} s$.
In addition, 
$$
I_{\Cc^{(e)}_{\overline{P}}} \cap K[x_A^{\varepsilon} :  
A \in \Ac(P), \varepsilon \in \{0,1\}^n]$$
is the toric ideal $I_{\Cc_{\overline{P}}}$.
For $\varepsilon=(\varepsilon_1,\ldots,\varepsilon_n) \in \{-1,0,1\}^n$, we write $\varepsilon^{+}:=(|\varepsilon_1|,\ldots,|\varepsilon_n|) \in \{0,1\}^n$.
We identify the variable $x^{\varepsilon^{+}}_A$ on $R[\Cc^{(e)}]$ with the variable $x_A$ on $R[\Cc]$.
It is known \cite[Proposition~1.11]{sturmfels1996} that there exists a nonnegative weight vector
$\wb_{\Cc} \in \RR^{|\Jc(P)|}$ such that ${\rm in}_{\wb_{\Cc}} (I_{\Cc_{\overline{P}}}) = {\rm in}_{<_{\Cc}}(I_{\Cc_{\overline{P}}})$.
Then we define the weight vector $\wb_{\Cc^{(e)}}$ on $R[\Cc^{(e)}]$ such that
the weight of each variable $x_A^\varepsilon$
with respect to $\wb_{\Cc^{(e)}}$ is the weight of the variable $x_A^{\epsilon^+}$ with respect to $\wb_{\Cc}$.
In addition, let $\wb_{\sharp}$ be
the weight vector on $R[\Cc^{(e)}]$ such that
the weight of each variable $x_A^\varepsilon$
with respect to $\wb_{\sharp}$ is $|A|$.
Fix any monomial order $\prec$ on $K[\Cc^{(e)}]$ as a tie-breaker.
Let $<_{\Cc^{(e)}}$ be a monomial order on $R[\Cc^{(e)}]$ such that $u <_{\Cc^{(e)}} v$
if and only if one of the following holds:
\begin{itemize}
	\item
	The weight of $u$ is less than that of $v$ with respect to $\wb_\sharp$;
	\item
	The weight of $u$ is the same as that of $v$ with respect to $\wb_\sharp$,
	and the weight of $u$ is less than that of $v$ with respect to $\wb_{\Cc^{(e)}}$;
	\item
	The weight of $u$ is the same as that of $v$ with respect to $\wb_\sharp$
	and $\wb_{\Cc^{(e)}}$, and $u \prec v$.
\end{itemize}
The following proposition was given in \cite[Theorem~1.3]{OTenriched}:

\begin{Proposition}[{\cite{OTenriched}}]
	\label{thm:grobner}
	Work with the same notation as above.
	Let $\Gc_{\Cc^{(e)}}$ be the set of all binomials
	\begin{equation*}
	\label{saisyo}
	x_{\max(I)}^{(\varepsilon_1,\ldots,\varepsilon_n)} x_{\max(J)}^{(\mu_1,\ldots,\mu_n)}
	-
	x_{\max(I) \setminus \{p\}}^{(\varepsilon_1,\ldots,\varepsilon_{p-1},0,\varepsilon_{p+1},\ldots,\varepsilon_{n})} x_{\max(J) \setminus \{p\}}^{(\mu_1,\ldots,\mu_{p-1},0,\mu_{p+1},\ldots,\mu_n)},
	\end{equation*}
	where $I, J \in \Jc(P)$, $\epsilon_p\neq \mu_p$, and $p \in \max(I) \cap \max(J)$,
together with all binomials
	\begin{equation*}
	\label{niban}
	x_{\max(I)}^{(\varepsilon_1,\dots, \varepsilon_n)} x_{\max(J)}^{(\varepsilon'_{1},\dots, \varepsilon'_{n})}
	-
	x_{\max (I \cup J)}^{(\mu_1,\dots, \mu_{n})} 
	x_{\max (I*J)}^{(\mu'_{1},\dots, \mu'_{n})},
	\end{equation*}
	where $I, J \in \Jc(P)$ with $I \nsim J$ and 
	\begin{itemize}
		\item[{\rm (a)}]
		For any $p \in \max(I) \cap \max(J)$,
		we have $\varepsilon_p =\varepsilon'_p =\mu_p=\mu'_p ${\rm ;}

\smallskip

		\item[{\rm (b)}]
		For any $p \in \max(I) \setminus \max (J)$, we have
$
\varepsilon_p
 = 
\left\{
\begin{array}{cc}
\mu_p & \mbox{if } p \in \max(I \cup J),\\
\mu_p' &  \mbox{if } p \in \max(I*J){\rm ;}
\end{array}
\right.
$

\smallskip

		\item[{\rm (c)}]
		For any $p \in \max(J) \setminus \max (I)$, we have
$
\varepsilon_p'
 = 
\left\{
\begin{array}{cc}
\mu_p & \mbox{if } p \in \max(I \cup J),\\
\mu_p' &  \mbox{if } p \in \max(I*J).
\end{array}
\right.
$
	\end{itemize}
	Then $\Gc_{\Cc^{(e)}}$ is a Gr\"obner basis of $I_{\Cc^{(e)}_{\overline{P}}}$ with respect to a monomial order $<_{\Cc^{(e)}}$.
	The initial monomial of each binomial is the first monomial.
	In particular, the initial ideal is generated by squarefree quadratic monomials
	which do not contain the variable $x_\emptyset^{\bf 0}$.
\end{Proposition}

By Lemma~\ref{genten_yowaku} and Proposition~\ref{thm:grobner},
we have the following immediately.

\begin{Corollary}[\cite{OTenriched}]
	Let $P$ be a finite poset on $[n]$.
	Then $\Cc^{(e)}_P$ is a reflexive polytope with a flag regular unimodular triangulation.
	Moreover, $K[\Cc^{(e)}_P]$ is a normal Gorenstein domain and Koszul.
\end{Corollary}

Next, we study Ehrhart polynomials of enriched chain polytopes.
Assume that $P$ is naturally labeled.
Let $\Omega^{(\ell)}_P(m)$ denote the number of left enriched $P$-partitions $f : P \to \ZZ$
with $|f(i)| \le m$.
Then $\Omega^{(\ell)}_P(m)$ is a polynomial in $m$ of degree $n$ and called the \textit{left enriched order polynomial} of $P$.

\begin{Proposition}[{\cite[Theorem~0.2]{OTenriched}}]
\label{enrichedchainpolytope:enrichedleft}
	Let $P$ be a naturally labeled finite poset on $[n]$.
	Then one has 
	\[
	L_{\Cc^{(e)}_P}(m) = \Omega_P^{(\ell)} (m).
	\]
\end{Proposition}

\section{Fundamental properties of enriched order polytopes}
In this section, we discuss some fundamental properties of enriched order polytopes.
First, we consider the set of lattice points in enriched order polytopes.
\begin{Lemma}
	\label{lem:latticepoint}
	Let $P$ be a finite poset on $[n]$. Then one has
	\[\Oc^{(e)}_P \cap \ZZ^n=
	\{
	\eb_{F_{\min}}^{\epsilon} + \eb_{F_{\comin}}
	:
	F \in \Fc(P), \epsilon \in \{-1,1\}^{|F_{\min}|}
	\}.
	\] 
In addition, the origin is the unique interior lattice point in $\Oc^{(e)}_P$.
\end{Lemma}

\begin{proof}
	Let $X = \{
	\eb_{F_{\min}}^{\epsilon} + \eb_{F_{\comin}}
	:
	F \in \Fc(P), \epsilon \in \{-1,1\}^{|F_{\min}|}
	\}$.
	It is enough to show that
	$\Oc^{(e)}_P \cap \ZZ^n \subset X.$
	Let $\xb = (x_1,\dots,x_n) \in \Oc^{(e)}_P \cap \ZZ^n$.
	Since $\Oc^{(e)}_P$ is the convex hull of $X$, 
	there exist $\ab_1,\dots, \ab_s \in X$ such that
	$\xb = \sum_{i=1}^s \lambda_i \ab_i $, where $\lambda_i > 0 $, $ \sum_{i=1}^s \lambda_i =1$.
	Then each $\ab_i$ is a $(-1,0,1)$-vector, and hence so is $\xb$.
	It is easy to see that $x_k = 1$ (resp.~$x_k = - 1$) if and only if $k$-th component of $\ab_i$ is equal to $1$ (resp.~$-1$) for all $i = 1,2,\ldots,s$.
	Suppose that $k <_P \ell$.
	If $x_k = 0$, then $|x_k| \le |x_\ell|$ and the equality holds if and only if $x_\ell =0$.
	Suppose that $|x_k| =1$.
	Then $k$-th component of $\ab_i$ is equal to $x_k $ for all $i = 1,2,\ldots,s$.
	Since each $\ab_i$ is a left enriched $P$-partition, 
	$\ell$-th component of $\ab_i$ is equal to $1$ for all $i = 1,2,\ldots,s$.
	Hence $x_\ell= 1$.
	In particular, $|x_k| = |x_\ell|$ and $x_\ell \ge 0$.
	Thus $\xb$ is a left enriched $P$-partition, that is, $\xb$ belongs to $X$.

Since $\Oc^{(e)}_P$ is an $n$-dimensional subpolytope of a cube $[-1,1]^n$,
it follows that each nonzero vector $\xb \in X$ belongs to the boundary of $\Oc^{(e)}_P$.
Suppose that the origin ${\bf 0} \in \RR^n$ belongs to the boundary of $\Oc^{(e)}_P$.
Then there exists a facet $\Fc$ of $\Oc^{(e)}_P$ which contains ${\bf 0}$.
Let ${\mathcal H} =\{ \yb \in \RR^n :\langle \ab,\yb \rangle = 0\}$ with ${\bf 0} \ne \ab = (a_1,\dots, a_n) \in \RR^n$ be the supporting hyperplane of $\Fc$ and let $P' = \{ i \in P :  a_i \ne 0\}$ $(\ne \emptyset)$ be a subposet of $P$.
We may assume that $i \in \max(P')$ satisfies $a_i >0$.
Let $F =\{ j \in P : i \le_P j \}$ be a filter of $P$.
Then $F_{\min} = \{i\}$ and hence
$\yb = \eb_{F_{\min}}^{(1)} + \eb_{F_{\comin}}$ satisfies $ \langle \ab,\yb \rangle = a_i > 0$
and $\yb' = \eb_{F_{\min}}^{(-1)} + \eb_{F_{\comin}}$ satisfies $ \langle \ab,\yb' \rangle = - a_i < 0$.
This contradicts that ${\mathcal H}$ is a supporting hyperplane of $\Oc^{(e)}_P$.
\end{proof}

Next, we consider lattice points in the dilated polytopes of an enriched order polytope.
The following example shows that, contrary to the case of order polytopes, the set of left enriched $P$-partitions $f : P \to \ZZ$ wtih $|f(i)| \le m$ does not always coincide with the set of lattice points $m\Oc^{(e)}_P \cap \ZZ^n$ if $m >1$.

\begin{Example}
\label{2chain}
Let $P$ be a poset on $\{1,2\}$ with $1 <_P 2$.	
Then the set of left enriched $P$-partitions $f : P \to \ZZ$ with $|f(i)| \le 2$
is 
$$\{
(0,0), (0, \pm 1), (0, \pm 2), 
(\pm 1, 1), (\pm 1, \pm 2), (\pm 2, 2)
\},
$$
and
$$2 \Oc^{(e)}_P \cap \ZZ^2
=
\{
(0,0), (0, \pm 1), (0, \pm 2), 
(\pm 1, 1), (\pm 1, 2),  (\pm 1 ,0), (\pm 2, 2)
\}
.$$
Thus two sets  are different.
On the other hand, the cardinality of each set is the same.
Moreover, it follows that $L_{\Oc^{(e)}_{P}}(m)=\Omega_P^{(\ell)} (m)=2m^2+2m+1$.
\end{Example}


\section{the toric ideals of enriched order polytopes}
In this section, we discuss the toric ideals of enriched order polytopes.
Let $P$ be a finite poset on $[n]$.
For a poset ideal $I$ of $P$, we set $I_{\max}:=\max(I)$ and $I_{\comax}:=I \setminus I_{\max}$.
Then lattice points in $\Oc^{(e)}_{\overline{P}}$ can be written by poset ideals of $P$:
	\[\Oc^{(e)}_{\overline{P}} \cap \ZZ^n=
	\{
	\eb_{I_{\max}}^{\epsilon} + \eb_{I_{\comax}}
	:
	I \in \Jc(P), \epsilon \in \{-1,1\}^{|I_{\max}|}
	\}.
	\] 
Contrary to the case of order polytopes, the enriched order polytopes $\Oc^{(e)}_{P}$ and $\Oc^{(e)}_{\overline{P}}$ are not always unimodularly equivalent.

\begin{Example}
\label{dualexample}
Let $P$ be the following poset on $\{1,2,3\}$:
		\newline
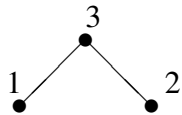
		\begin{center}
	\begin{picture}(150,40)(10,30)
	\put(70,45){$1$}
	\put(130,45){$2$}
	\put(100,70){$3$}
	\put(75,40){\circle*{5}}
	\put(125,40){\circle*{5}}
	\put(100,65){\circle*{5}}
	\put(100,65){\line(1,-1){25}}
	\put(100,65){\line(-1,-1){25}}
	
	\end{picture}		\end{center}
Then $\Oc^{(e)}_{P}$ has 5 facets and $\Oc^{(e)}_{\overline{P}}$ 
has 6 facets.
Thus $\Oc^{(e)}_{P}$ and $\Oc^{(e)}_{\overline{P}}$ are not unimodularly equivalent.
On the other hand, it follows that
$$L_{\Oc^{(e)}_{P}}(m)= L_{\Oc^{(e)}_{\overline{P}}}(m)=
\binom{m+3}{3} + 7 \binom{m+2}{3}+ 7 \binom{m+1}{3}+  \binom{m}{3}
.$$
\end{Example}

Now, we consider the toric ideals $I_{\Oc^{(e)}_{\overline{P}}}$.
Let
$R[\Oc^{(e)}]$ be the polynomial ring over $K$ in variables $x_I^{\epsilon}$, where $I \in \Jc(P)$ and $\epsilon=(\epsilon_1,\ldots,\epsilon_n) \in \{-1,0,1\}^{n}$ with
\[
\epsilon_i=
\begin{cases}
1 \mbox{ or } -1 & (i \in \max(I));\\
1 & (i \in \comax(I));\\
0 & (i \notin I).
\end{cases}
\]
Then the toric ideal $I_{\Oc^{(e)}_{\overline{P}}}$  is the kernel of a ring homomorphism 
$\pi_{\Oc^{(e)}}: R[\Oc^{(e)}] \rightarrow K[\tb^{\pm1}, s]$
defined by
$\pi_{\Oc^{(e)}}(x_{I}^{\epsilon}) = t_{1}^{\varepsilon_1}
\dots t_{n}^{\varepsilon_n}  s$.
In addition, 
$$
I_{\Oc^{(e)}_{\overline{P}}} \cap K[x_{I}^{\epsilon} :  
I \in \Jc(P), \epsilon \in \{0,1\}^n]$$
is the toric ideal $I_{\Oc_{\overline{P}}}$.
We define a reverse lexicographic order $<_{\Oc^{(e)}}$ on $R[\Oc^{(e)}]$ such that $x^{\epsilon}_I <_{\Oc^{(e)}} x^{\mu}_J$ if $I \subsetneq J$.

\begin{Theorem}
\label{enrichedorderGB}
	Work with the same notation as above.
	Let $\Gc_{\Oc^{(e)}}$ be the set of all binomials
	\begin{equation}
	\label{eorder:ichiban}
	x_{I}^{(\varepsilon_1,\dots, \varepsilon_n)} x_{J}^{(\mu_1,\dots, \mu_n)}
	-
	x_{I \setminus \{p\}}^{(\varepsilon_1,\dots, \varepsilon_{p-1}, 0, \varepsilon_{p+1}, \dots,\varepsilon_n)} x_{J \setminus \{p\}}^{(\mu_1,\dots, \mu_{p-1}, 0,  \mu_{p+1}, \dots, \mu_n)},
	\end{equation}
	where $I, J \in \Jc(P)$, $\varepsilon_p \neq \mu_p$, and $p \in \max(I) \cap \max(J)$, together with all binomials
	\begin{equation}
	\label{eorder:niban}
	x_{I}^{(\varepsilon_1,\dots, \varepsilon_n)} x_{J}^{(\varepsilon'_{1},\dots, \varepsilon'_{n})}
	-
	x_{I \cup J}^{(\mu_1,\dots, \mu_{n})} 
	x_{I \cap J}^{(\mu'_{1},\dots, \mu'_{n})},
	\end{equation}
	where $I, J \in \Jc(P)$ with $I \nsim J$, and 
	\begin{itemize}
		\item[{\rm (a)}]
		For any $p \in \max(I) \cap \max(J)$,
		we have $\varepsilon_p =\varepsilon'_p =\mu_p=\mu'_p${\rm ;}

\smallskip

		\item[{\rm (b)}]
		For any $p \in \max(I) \setminus \max (J)$, we have
$
\varepsilon_p
 = 
\left\{
\begin{array}{cc}
\mu_p & \mbox{if } p \in \max(I \cup J),\\
\mu_p' &  \mbox{if } p \in \max(I \cap J){\rm ;}
\end{array}
\right.
$

\smallskip

		\item[{\rm (c)}]
		For any $p \in \max(J) \setminus \max (I)$, we have
$
\varepsilon_p'
 = 
\left\{
\begin{array}{cc}
\mu_p & \mbox{if } p \in \max(I \cup J),\\
\mu_p' &  \mbox{if } p \in \max(I \cap J).
\end{array}
\right.
$
	\end{itemize}
Then $\Gc_{\Oc^{(e)}}$ is a Gr\"obner basis of $I_{\Oc^{(e)}_{\overline{P}}}$ with respect to a monomial order $<_{\Oc^{(e)}}$.
The initial monomial of each binomial is the first monomial.
In particular, the initial ideal is generated by squarefree quadratic monomials
which do not contain the variable $x_\emptyset^{\bf 0}$.
\end{Theorem}

\begin{proof}
	It is easy to see that any binomial of type (\ref{eorder:ichiban}) belongs to $I_{\Oc^{(e)}_{\overline{P}}}$.
By Lemma \ref{fourconditions}, 
it follows that any binomial of type (\ref{eorder:niban}) belongs to $I_{\Oc^{(e)}_{\overline{P}}}$.
	Hence $\Gc_{\Oc^{(e)}}$ is a subset of $I_{\Oc^{(e)}_{\overline{P}}}$.
Moreover, the initial monomial of each binomial is the first monomial.
	Assume that $\Gc_{\Oc^{(e)}}$ is not a Gr\"obner basis of $I_{\Oc^{(e)}_{\overline{P}}}$
	with respect to $<_{\Oc^{(e)}}$.
	Let 
$${\rm in}(\Gc_{\Oc^{(e)}}) = \left( {\rm in}_{<_{\Oc^{(e)}}} (g)  : g \in \Gc_{\Oc^{(e)}} \right)
.$$
By \cite[Theorem~3.11]{binomialideals}, there exists a non-zero irreducible homogeneous binomial $f = u-v \in I_{\Oc^{(e)}_{\overline{P}}}$
	such that neither $u$ nor $v$ belongs to ${\rm in}(\Gc_{\Oc^{(e)}})$.
	For $I, J \in \Jc(P)$ and $\varepsilon, \mu \in  \{-1,0,1\}^n$, if $i \in \max(I) \cap \max(J)$ 
satisfies $\varepsilon_i \neq \mu_i$, then $x_I^{\varepsilon} x_J^{\mu} \in {\rm in}(\Gc_{\Oc^{(e)}})$.
	On the other hand, for $I, J \in \Jc(P)$ with $I \nsim J$ and for $\varepsilon, \mu \in  \{-1,0,1\}^n$, if 
$\varepsilon_p=\mu_p$ for any $p \in \max(I) \cap \max(J)$, then $x_I^{\varepsilon} x_J^{\mu} \in {\rm in}(\Gc_{\Oc^{(e)}})$.
	Hence $u$ and $v$ are of the form
	$$
	u= 
x_{I_1}^{\varepsilon^{(1)}}x_{I_2}^{\varepsilon^{(2)}} \cdots x_{I_r}^{\varepsilon^{(r)}},
	\ 
	v=x_{J_1}^{\mu^{(1)}}x_{J_2}^{\mu^{(2)}} \cdots x_{J_r}^{\mu^{(r)}},
	$$
	where $I_k, J_k \in \Jc(P)$ and 
$
\varepsilon^{(k)} = (\varepsilon_1^{(k)},\dots, \varepsilon_n^{(k)}), 
\mu^{(k)}= (\mu_1^{(k)}, \dots, \mu_n^{(k)}) \in \{-1,0,1\}^n$ for $k =1,2,\dots, r$ such that 
	\begin{itemize}
		\item[(a)] $I_1 \subset \dots \subset I_r$ and $J_1 \subset \dots \subset J_r$;
		\item[(b)] For any $p$ and $q$, and for any $i \in \max(I_p) \cap \max(I_q)$, we obtain $\varepsilon^{(p)}_i=\varepsilon^{(q)}_i$; 
		\item[(c)] For any $p$ and $q$, and for any $j \in \max(J_p) \cap \max(J_q)$, we obtain $\mu^{(p)}_j=\mu^{(q)}_j$.
	\end{itemize}
	Since $u$ and $v$ satisfy conditions (b) and (c) and since $f$ belongs to $I_{\Oc^{(e)}_{\overline{P}}}$,
	it then follows that 
	$\max(I_r)=\max(J_r)$ and $\varepsilon_r=\mu_r$.
	Hence one has $x_{I_r}^{(\varepsilon_r)}
	=
	x_{J_r}^{(\mu_r)}$.
	This contradicts the assumption that $f$ is irreducible.
\end{proof}

By Lemma~\ref{genten_yowaku} and Theorem~\ref{enrichedorderGB},
we have the following immediately.

\begin{Corollary}
\label{maincorollary}
	Let $P$ be a finite poset on $[n]$.
	Then $\Oc^{(e)}_P$ is a reflexive polytope with a flag regular unimodular triangulation.
	Moreover, $K[\Oc^{(e)}_P]$ is a normal Gorenstein domain and Koszul.
\end{Corollary}

\begin{Theorem}
\label{lasttheorem}
	Work with the same notation as above. 
	Then one has
	\[
	\dfrac{R[\Oc^{(e)}]}{{\rm in}_{<_{\Oc^{(e)}}}(I_{\Oc^{(e)}_{\overline{P}}}) } \cong \dfrac{R[\Cc^{(e)}]}{{\rm in}_{<_{\Cc^{(e)}}}(I_{\Cc^{(e)}_{\overline{P}}})}.
	\]
	Furthermore, we obtain
	\[
	L_{\Oc^{(e)}_{\overline{P}}}(m)=L_{\Cc^{(e)}_{\overline{P}}}(m)=L_{\Cc^{(e)}_{{P}}}(m)=L_{\Oc^{(e)}_{P}}(m).
	\]
\end{Theorem}

\begin{proof}
From Theorem \ref{enrichedorderGB},
${\rm in}_{<_{\Oc^{(e)}}}(I_{\Oc^{(e)}_{\overline{P}}})$
is generated by all monomials
$$
	x_I^{(\varepsilon_1,\ldots,\varepsilon_n)} x_J^{(\mu_1,\ldots,\mu_n)},
$$
where $I, J \in \Jc(P)$, $\epsilon_p\neq \mu_p$,  and $p \in \max(I) \cap \max(J)$
together with all monomials
$$
x_I^{(\varepsilon_1,\dots, \varepsilon_n)} x_J^{(\varepsilon'_{1},\dots, \varepsilon'_{n})},
$$
where  $I, J \in \Jc(P)$ with $I \nsim J$ and
$\varepsilon_p =\varepsilon'_p$ for each $p \in \max(I) \cap \max(J)$.
Moreover, from Proposition \ref{thm:grobner},
${\rm in}_{<_{\Cc^{(e)}}}(I_{\Cc^{(e)}_{\overline{P}}})$
is generated by all monomials
$$
	x_{\max(I)}^{(\varepsilon_1,\ldots,\varepsilon_n)} x_{\max(J)}^{(\mu_1,\ldots,\mu_n)},
$$
where $I, J \in \Jc(P)$, $\epsilon_p\neq \mu_p$,  and $p \in \max(I) \cap \max(J)$
together with all monomials
$$
x_{\max(I)}^{(\varepsilon_1,\dots, \varepsilon_n)} x_{\max(J)}^{(\varepsilon'_{1},\dots, \varepsilon'_{n})},
$$
where  $I, J \in \Jc(P)$ with $I \nsim J$ and
$\varepsilon_p =\varepsilon'_p$ for each $p \in \max(I) \cap \max(J)$.
	Hence it follows that the map $x_{I}^{(\varepsilon_1,\ldots,\varepsilon_n)} \mapsto x_{\max(I)}^{(\varepsilon_1',\ldots,\varepsilon_n')}$, where $\varepsilon_i' = \varepsilon_i$ 
for $i \in \max (I)$ and  $\varepsilon_i' = 0$ for $i \notin \max(I)$, induces an isomorphism for the first claim.
By the argument in the last part of Proof of Proposition \ref{iso},
we have 
$
	L_{\Oc^{(e)}_{\overline{P}}}(m)=L_{\Cc^{(e)}_{\overline{P}}}(m)$
and 
$L_{\Oc^{(e)}_{P}}(m)
=L_{\Cc^{(e)}_{{P}}}(m).
$
Since $\Cc^{(e)}_{{P}} = \Cc^{(e)}_{\overline{P}}$,
the second claim follows.
\end{proof}

By Proposition~\ref{enrichedchainpolytope:enrichedleft} and Theorem~\ref{lasttheorem}, we have Theorem~\ref{thm:ehrhart}.

\section{Facets of enriched order polytopes and enriched chain polytopes}

Theorem~\ref{thm:ehrhart} implies the existence of a bijection between $m\Oc^{(e)}_P \cap \ZZ^n$ and  
$m\Cc^{(e)}_P \cap \ZZ^n$.
Towards such a bijection, in this section, we consider an elementary  
geometric property, the facet representations of enriched order polytopes and enriched chain polytopes.

Let $P$ be a finite poset on $[n]$.
Given elements $i, j$ of $P$,
we say that $j$ {\em covers} $i$
if $i < j$ and there exists no $k \in P$ such that $i< k < j$.
If $j$ covers $i$ in $P$, then we write $i \lessdot j$.
A {\em chain} of $P$ is a totally ordered subset of $P$.
A chain of the form $i_1 \lessdot i_2 \lessdot \cdots \lessdot i_r$ is called a {\em saturated} chain.
A saturated chain $i_1 \lessdot i_2 \lessdot \cdots \lessdot i_r$ is called {\em maximal}
if $i_1 \in \min(P)$ and $i_r \in \max(P)$.
First, we give the facet representations of enriched chain polytopes
which easily follows from \cite[Lemma~1.1]{OTenriched}
and the facet representations of chain polytopes \cite{twoposetpolytopes}.

\begin{Proposition}
\label{facetsofC}
Let $P$ be a finite poset on $[n]$.
Then $\Cc^{(e)}_P \subset \RR^n$ is the solution set of the linear inequalities
$$
\sum_{j=1}^{r} \varepsilon_j x_{i_j} \le1 ,
$$
where $i_1 \lessdot i_2 \lessdot \cdots \lessdot i_r$ 
is a maximal chain of $P$, and $\varepsilon_j \in \{1,-1\}$.
In addition, each of the above inequalities is facet defining.
\end{Proposition}

On the other hand, the facet representations of enriched order polytopes
are as follows.

\begin{Theorem}
\label{facetsofO}
Let $P$ be a finite poset on $[n]$.
Then $\Oc^{(e)}_P \subset \RR^n$ is the solution set of the following linear inequalities{\rm :}
	\begin{itemize}
		\item[{\rm (a)}] $2^{r-1} x_{i_1} - \sum_{j=2}^{r} 2^{r-j} x_{i_j} \le 1$, where $i_1 \lessdot i_2 \lessdot \cdots \lessdot i_r$ 
is a saturated chain of $P$ with $i_r \in \max(P)${\rm ;}
		\item[{\rm (b)}] $ - \sum_{j=1}^{r} 2^{r-j} x_{i_j} \le1$, where $i_1 \lessdot i_2 \lessdot \cdots \lessdot i_r$ 
is a maximal chain of $P$.
	\end{itemize}
In addition, each of the above inequalities is facet defining.
\end{Theorem}

\begin{proof}
The proof is induction on $n$.
If $n=1$, then the assertion is trivial.
Assume $n \ge 2$.

Let ${\mathcal Q} \subset \RR^n$ be the solution set of the above linear inequalities.
Since 
$2^{s-1} - \sum_{j=2}^{s} 2^{s-j} = 1$
holds for any positive integer $s$,
it is easy to see that
$\eb_{F_{\min}}^{\epsilon} + \eb_{F_{\comin}}$
satisfies (a) and (b) for any filter $F$ of $P$, and
for any $\epsilon \in \{-1,1\}^{|F_{\min}|}$.
Since $\Oc^{(e)}_P$ is the convex hull of such vectors,
we have ${\mathcal Q} \supset \Oc^{(e)}_P$.
In order to prove ${\mathcal Q} \subset \Oc^{(e)}_P$,
let ${\bf x} = (x_1, \dots, x_n) \in {\mathcal Q}$.
First, we will show that $|x_i| \leq 1$ for each $i \in [n]$.
Let $i= i_1 \lessdot i_2 \lessdot \cdots \lessdot i_r$ be a saturated chain of $P$
with $i_r \in \max(P)$.
Then ${\bf x}$ satisfies the following $r$ inequalities:
\begin{align}
2^{r-1} x_{i_1} - \sum_{j=2}^{r} 2^{r-j} x_{i_j} \le&\  1, \tag{$a_1$}\\
2^{r-2} x_{i_2} - \sum_{j=3}^{r} 2^{r-j} x_{i_j} \le& \ 1,\tag{$a_2$}\\
  \vdots & \nonumber \\
 x_{i_r} \le&\  1.  \tag{$a_r$}
\end{align}
If $r=1$, then $x_i \le 1$ is trivial.
Let $r \ge 2$.
Then the inequality given by a linear combination
$(a_1) +  (a_2) + 2 (a_3) + \dots + 2^{r-2}(a_r)$
of the above inequalities
is $2^{r-1} x_{i_1}  \le 2^{r-1}$, and hence
$x_i = x_{i_1} \le 1$.
Suppose that $i$ belongs to a maximal chain $i_1 \lessdot i_2 \lessdot \cdots \lessdot i_r$, say, $i = i_k$.
Then ${\bf x}$ satisfies $(a_1), \dots, (a_r)$ above and
\begin{align}
 - \sum_{j=1}^{r} 2^{r-j} x_{i_j} \le& 1. \tag{$b_1$}
\end{align}
Then the inequality given by a linear combination
$$(b_1)+ (a_1) +  2 (a_2) +  \dots + 2^{k-2}(a_{k-1}) + 2^{k-1}(a_{k+1}) + \dots +  2^{r-2}(a_r) $$
of the above inequalities
is $- 2^{r-1} x_{i_k}  \le 2^{r-1}$, and hence
we have $x_i = x_{i_k} \ge -1$.

We now prove that ${\bf x}$ belongs to $\Oc^{(e)}_P$
by induction on $n$.
Suppose that $x_i = 0$ for some $i \in \min(P)$.
Then $(x_1,\dots,x_{i-1},x_{i+1},\dots,x_n) \in \RR^{n-1}$ satisfies inequalities (a) and (b) for 
the subposet $P \setminus \{i\}$ of $P$.
By the assumption of induction,  $(x_1,\dots,x_{i-1},x_{i+1},\dots,x_n)$
belongs to $\Oc^{(e)}_{P \setminus \{i\}}$.
It then follows that ${\bf x}$ belongs to $\Oc^{(e)}_P$.
Thus we may assume that $x_i \neq 0$ for any $i \in \min(P)$.
Let $
\lambda = \min \{ |x_i| :  i \in \min(P)  \}
$.
Note that $0 < \lambda \le 1$.
Let 
$$
{\bf y} = (y_1,\dots,y_n) = {\bf x} - \lambda
( \eb_{F_{\min}}^{\epsilon} + \eb_{F_{\comin}}),
$$
where $F = [n]$, and
 $\epsilon \in \{-1,1\}^{|F_{\min} |}$ corresponds to the sign of $x_i$ for each
$ i \in \min(P) = F_{\min} $.
We now show that
the vector ${\bf y}$ satisfies 
	\begin{itemize}
		\item[{\rm (c)}] $2^{r-1} y_{i_1} - \sum_{j=2}^{r} 2^{r-j} y_{i_j} \le 1 - \lambda$
, where $i_1 \lessdot i_2 \lessdot \cdots \lessdot i_r$ is a saturated chain of $P$ with $i_r \in \max(P)${\rm ;}
		\item[{\rm (d)}]  $-2^{r-1} y_{i_1} - \sum_{j=2}^{r} 2^{r-j} y_{i_j} \le 1 - \lambda$
, where $i_1 \lessdot i_2 \lessdot \cdots \lessdot i_r$ is a maximal chain of $P$.
	\end{itemize}

\bigskip

\noindent
\underline{Inequality (c).}
If either $x_{i_1} >0$ or $i_1 \notin \min(P)$ holds,
then
$$
2^{r-1} y_{i_1} - \sum_{j=2}^{r} 2^{r-j} y_{i_j} 
=
2^{r-1} (x_{i_1}-\lambda) - \sum_{j=2}^{r} 2^{r-j} (x_{i_j}-\lambda)
\le 1 - \lambda.$$
If $x_{i_1} < 0$ and $i_1 \in \min(P)$, then
$\lambda + x_{i_1} \le 0$ and hence
\begin{eqnarray*}
2^{r-1} y_{i_1} - \sum_{j=2}^{r} 2^{r-j} y_{i_j} 
&=&
2^{r-1} (x_{i_1} + \lambda) - \sum_{j=2}^{r} 2^{r-j} (x_{i_j}-\lambda)\\
&=&
(2^r-1) \lambda +2^{r-1} x_{i_1} - \sum_{j=2}^{r} 2^{r-j}  x_{i_j}\\
&=&
2^r (\lambda + x_{i_1} ) - \lambda - 2^{r-1} x_{i_1} - \sum_{j=2}^{r} 2^{r-j}  x_{i_j}\\
&\le& 1 - \lambda.
\end{eqnarray*}

\bigskip

\noindent
\underline{Inequality (d).}
If $x_{i_1} < 0$, then we have
$$
- \sum_{j=1}^{r} 2^{r-j} y_{i_j} 
=
-2^{r-1} (x_{i_1}+\lambda) - \sum_{j=2}^{r} 2^{r-j} (x_{i_j}-\lambda)
\le 1 - \lambda.$$
If $x_{i_1} > 0$, then $\lambda - x_{i_1} \le 0$ and hence
\begin{eqnarray*}
- \sum_{j=1}^{r} 2^{r-j} y_{i_j} 
&=&
-2^{r-1} (x_{i_1} - \lambda) - \sum_{j=2}^{r} 2^{r-j} (x_{i_j}-\lambda)\\
&=&
(2^r-1) \lambda  -2^{r-1} x_{i_1} - \sum_{j=2}^{r} 2^{r-j}  x_{i_j}\\
&=&
2^r (\lambda - x_{i_1} ) - \lambda + 2^{r-1} x_{i_1} - \sum_{j=2}^{r} 2^{r-j}  x_{i_j}\\
&\le& 1 - \lambda.
\end{eqnarray*}

\bigskip

If $\lambda =1$, then we have ${\bf y}={\bf 0}$
by inequalities (c) and (d).
Hence ${\bf x} =
\eb_{F_{\min}}^{\epsilon} + \eb_{F_{\comin}}
\in  \Oc^{(e)}_P$.
If $\lambda \neq 1$, then
$\frac{1}{1-\lambda} {\bf y}$ 
belongs to ${\mathcal Q}$
by inequalities (c) and (d).
From the definition of $\lambda$, 
there exists $i \in \min(P)$ such that $y_i=0$.
By the assumption of induction, $\frac{1}{1-\lambda} {\bf y}$ 
belongs to $\Oc^{(e)}_P$, and hence ${\bf y}$
belongs to $(1 - \lambda) \Oc^{(e)}_P$.
Thus ${\bf x} = \lambda
( \eb_{F_{\min}}^{\epsilon} + \eb_{F_{\comin}})
+
{\bf y}$ belongs to $\lambda \Oc^{(e)}_P + (1 - \lambda) \Oc^{(e)}_P
= \Oc^{(e)}_P$.

Finally, we will prove that each of inequalities (a) and (b)
 is facet defining.
Let 
$$
{\mathcal H}_{i_1   i_2   \cdots   i_r}^+
=
\left\{(x_1,\dots, x_n) \in \RR^n : 
2^{r-1} x_{i_1} - \sum_{j=2}^{r} 2^{r-j} x_{i_j} = 1\right\},
$$
where $i_1 \lessdot i_2 \lessdot \cdots \lessdot i_r$ 
is a saturated chain of $P$ with $i_r \in \max(P)$, and let
$$
{\mathcal H}_{i_1   i_2   \cdots   i_r}^-
=
\left\{(x_1,\dots, x_n) \in \RR^n :  - \sum_{j=1}^{r} 2^{r-j} x_{i_j}  = 1\right\},
$$
where $i_1 \lessdot i_2 \lessdot \cdots \lessdot i_r$ 
is a maximal chain of $P$.
It is enough to show that 
$$\dim (\Oc^{(e)}_P \cap {\mathcal H}_{i_1   i_2   \cdots   i_r}^+)= \dim (\Oc^{(e)}_P \cap {\mathcal H}_{i_1   i_2   \cdots   i_r}^-) = n-1.$$

Let $i_1 \lessdot i_2 \lessdot \cdots \lessdot i_r$ 
be a saturated chain of $P$ with $i_r \in \max(P)$.
If $\min(P) = \{i_1\}$, then let $i = i_1$.
If $\min(P) \neq \{i_1\}$, then let $i$ be an arbitrary element in $\min (P) \setminus \{i_1\}$.
Note that, if $\min(P) = \{i_1\}$, then 
$i_2 \lessdot i_3 \lessdot \cdots \lessdot i_r$
is a maximal chain of $P \setminus \{i\}$.
Let ${\mathcal H}' = \{ (x_1, \dots, x_n) \in \RR^n : x_i =0 \}$.
Then
\begin{eqnarray*}
 & & \Oc^{(e)}_P \cap {\mathcal H}_{i_1   i_2   \cdots   i_r}^+
\cap {\mathcal H}' \\
&=&
\left\{
\begin{array}{lc}
(\Oc^{(e)}_P \cap {\mathcal H}' ) \cap 
\{(x_1,\dots, x_n) \in \RR^n : x_i = 0, 
2^{r-2} x_{i_1} - \sum_{j=2}^{r} 2^{r-j} x_{i_j} = 1\} & \mbox{ if } i \neq i_1,\\
\\
(\Oc^{(e)}_P \cap {\mathcal H}' ) \cap 
\{(x_1,\dots, x_n) \in \RR^n : x_i = 0, 
- \sum_{j=2}^{r} 2^{r-j} x_{i_j} = 1\} &  \mbox{ if } i = i_1
\end{array}
\right.
\end{eqnarray*}
is unimodularly equivalent to a facet of $\Oc^{(e)}_{P \setminus \{i\}} $ by the assumption of induction.
Hence $\dim (\Oc^{(e)}_P \cap {\mathcal H}_{i_1   i_2   \cdots   i_r}^+
\cap {\mathcal H}' ) = n-2$.
Since $(1,\dots,1) \in \RR^n$ belongs to
$(\Oc^{(e)}_P \cap {\mathcal H}_{i_1   i_2   \cdots   i_r}^+) \setminus {\mathcal H}' $,
we have $\dim (\Oc^{(e)}_P \cap {\mathcal H}_{i_1   i_2   \cdots   i_r}^+)= n-2+1 = n-1.$
On the other hand, for a maximal chain 
 $i_1 \lessdot i_2 \lessdot \cdots \lessdot i_r$ of $P$,
let $i = i_1$ if $\min(P) = \{i_1\}$, and let $i$ be an arbitrary element in $\min (P) \setminus \{i_1\}$ otherwise.
Note that, if $\min(P) = \{i_1\}$, then 
$i_2 \lessdot i_3 \lessdot \cdots \lessdot i_r$
is a maximal chain of $P \setminus \{i\}$.
Then
\begin{eqnarray*}
 & & \Oc^{(e)}_P \cap {\mathcal H}_{i_1   i_2   \cdots   i_r}^-
\cap {\mathcal H}' \\
&=&
\left\{
\begin{array}{lc}
(\Oc^{(e)}_P \cap {\mathcal H}' ) \cap 
\{(x_1,\dots, x_n) \in \RR^n : x_i = 0, 
 - \sum_{j=1}^{r} 2^{r-j} x_{i_j}  = 1\} & \mbox{ if } i \neq i_1,\\
\\
(\Oc^{(e)}_P \cap {\mathcal H}' ) \cap 
\{(x_1,\dots, x_n) \in \RR^n : x_i = 0, 
- \sum_{j=2}^{r} 2^{r-j} x_{i_j} = 1\} &  \mbox{ if } i = i_1
\end{array}
\right.
\end{eqnarray*}
is unimodularly equivalent to a facet of $\Oc^{(e)}_{P \setminus \{i\}} $ by the assumption of induction.
Hence $\dim (\Oc^{(e)}_P \cap {\mathcal H}_{i_1   i_2   \cdots   i_r}^-
\cap {\mathcal H}' ) = n-2$.
Since $(1,\dots,1) - 2 {\bf e}_i \in \RR^n$ belongs to
$(\Oc^{(e)}_P \cap {\mathcal H}_{i_1   i_2   \cdots   i_r}^-) \setminus {\mathcal H}' $,
we have $\dim (\Oc^{(e)}_P \cap {\mathcal H}_{i_1   i_2   \cdots   i_r}^-)= n-2+1 = n-1,$
as desired.
\end{proof}

Given a polytope $\Pc$ of dimension $n$, let $f_{n-1} (\Pc)$
be the number of the facets of $\Pc$.
It is known \cite[Corollary~1.2]{uniequiv} that
$f_{n-1} ( \Oc_P) \leq f_{n-1} ( \Cc_P)$
for any poset $P$.

\begin{Corollary}
\label{numberoffacets}
Let $P$ be a finite poset on $[n]$.
Then we have the following{\rm :}
\begin{itemize}
\item[{\rm (a)}]
Let $sc(P)$ (resp.~$mc(P)$) be the number of saturated (resp.~maximal) chains 
of $P$ that contains a maximal element of $P$.
Then $f_{n-1} ( \Oc^{(e)}_P) = sc(P) + mc(P)$.
\item[{\rm (b)}]
Let $mc_\ell(P)$ be the number of maximal chains of $P$ of length $\ell$.
Then $f_{n-1} (\Cc^{(e)}_P)=\sum_{\ell=0}^{n-1} 2^{\ell+1} mc_\ell(P)$.
\end{itemize}
Moreover, we have $f_{n-1} ( \Oc^{(e)}_P) \leq f_{n-1} ( \Cc^{(e)}_P)$.
\end{Corollary}

\begin{proof}
The formulas of the number of facets follows from Proposition~\ref{facetsofC} and Theorem~\ref{facetsofO}.
Each maximal chain of $P$ of length $\ell$ contains exactly
$\ell+1$ saturated chains of $P$ that contains a maximal element of $P$.
Since $\ell + 2 \le 2^{\ell+1}$ for any integer $\ell \ge 0$, we have 
$sc(P) + mc(P) \leq \sum_{\ell=0}^{n-1} 2^{\ell+1} mc_\ell(P)$.
\end{proof}

In \cite[Lemma 3.8]{HLLMT}, tight upper bounds for 
$f_{n-1}(\Oc_P) $ and $f_{n-1}(\Cc_P) $ are given.
Given an integer $n\ge 2$, let 
$$
\mu_n=
\left\{
\begin{array}{cl}
3^k& \mbox{ if } n =3k,\\
4 \cdot 3^{k -1} & \mbox{ if } n =3k+1,\\
2 \cdot 3^k& \mbox{ if } n=3k+2.
\end{array}
\right.
$$
It is known \cite[Theorem 1]{MM} that $\mu_n$ is the maximum number of cliques
possible in a graph with $n$ vertices.

\begin{Proposition}[{\cite[Lemma 3.8]{HLLMT}}]
Let $P$ be a finite poset on $[n]$ with $n \ge 5$.
Then we have $f_{n-1}(\Cc_P) \le \mu_n + n$, and 
$f_{n-1}(\Oc_P) \le \lfloor \frac{n+1}{2} \rfloor ( n - \lfloor \frac{n+1}{2} \rfloor  ) +n.
$
In addition, both upper bounds are tight.
\end{Proposition}

We give tight upper bounds for the number of facets of enriched order and chain polytopes.

\begin{Proposition}
\label{upperbound}
Let $P$ be a finite poset on $[n]$.
Then we have $f_{n-1} (\Cc^{(e)}_P) \le  2^n$
and
$$
f_{n-1}(\Oc^{(e)}_P) \le 
\left\{
\begin{array}{ll}
2n  & \mbox{ if } n = 1,2,3,\\
\\
\frac{47}{2} \cdot 3^{k-2} -\frac{3}{2} & \mbox{ if } n = 3k \ (k \ge 2),\\
\\
\frac{23}{2} \cdot 3^{k-1} -\frac{3}{2} &\mbox{ if }  n = 3k+1 \ (k \ge 1),\\
\\
\frac{11}{2} \cdot 3^k -\frac{3}{2}&\mbox{ if }  n = 3k+2 \ (k \ge 1).
\end{array}
\right.
$$
In addition, both upper bounds are tight.
\end{Proposition}

\begin{proof}
The proof for $\Cc^{(e)}_P$ is induction on $n$.
If $n=1$, then $\Cc^{(e)}_P$ has two facets.
Let $n \ge 2$ and let $M$ be the set of all minimal elements of $P$.
If $|M| = m$, then we have
$$
f_{n-1} (\Cc^{(e)}_P) \le 2 m f_{n-m-1} (\Cc^{(e)}_{P \setminus M})  \le  2^{n-m+1} m \le 2^n
$$
by the assumption of induction.
Note that $f_{n-1} (\Cc^{(e)}_P) = 2^n$ if $P$ is a chain.

By explicit computation, 
for $n=1, 2,3,4$, the maximum value of the number of facets 
of $\Oc^{(e)}_P$ is $2$, $4$, $6$, $10$, respectively.
(Note that $f_{n-1} (\Oc^{(e)}_P)  = 2n$ if $P$ is an antichain.)
Thus the assertion for $\Oc^{(e)}_P$ holds for $n \le 4$.
Assume $n \ge 5$.
Let $P$ be a poset on $[n]$.
Let $P_1 = P$ and 
let $M_1$ be the set of all maximal elements of $P_1$.
If $P_1$ is not an antichain, then let $P_2 = P_1 \setminus M_1$ and let $M_2$ be the set of all maximal elements of $P_2$.
In general, if $P_i$ is not an antichain, then
$P_{i+1} = P_i \setminus M_i$ and let $M_{i+1}$ be the set of all maximal elements of $P_{i+1}$.
By this procedure, we get a sequence of posets
$P_1, \dots, P_r$ such that $P_r$ is an antichain.
Then we have
$$
f_{n-1} (\Oc^{(e)}_P) \le 
|M_1| + |M_1| |M_2|+ \cdots + 
|M_1| |M_2| \cdots |M_{r-1}|
+ 
2 |M_1| |M_2| \cdots |M_r|.
$$
We show that 
\begin{equation}
\label{ue}
\max
\left\{
2  m_1 m_2  \cdots m_r
+
\sum_{j=1}^{r-1}
\prod_{k=1}^j m_k
:
1 \le r \le n, \sum_{j=1}^r m_j = n, 1 \le m_i \in \ZZ 
\right\}
\end{equation}
is equal to 
$$
\left\{
\begin{array}{ll}
\frac{47}{2} \cdot 3^{k-2} -\frac{3}{2} & \mbox{ if } n = 3k,\\
\\
\frac{23}{2} \cdot 3^{k-1} -\frac{3}{2} &\mbox{ if }  n = 3k+1,\\
\\
\frac{11}{2} \cdot 3^k -\frac{3}{2}&\mbox{ if }  n = 3k+2,
\end{array}
\right.
$$
for $n \ge 5$.
Suppose that $m_1, \dots, m_r$, where 
 $1 \le r \le n$, $\sum_{j=1}^r m_j = n$, and $1 \le m_i \in \ZZ$ give the maximum value of (\ref{ue}). 
If $m_i < m_{i+1}$ for some $i$, then
$$
2  m_1 m_2  \cdots m_r
+
\sum_{j=1}^{r-1}
\prod_{k=1}^j m_k
<
2  m_1' m_2'  \cdots m_r'
+
\sum_{j=1}^{r-1}
\prod_{k=1}^j m_k',
$$
where $(m_i',m_{i+1}') = (m_{i+1},m_i)$ and
$m_k' = m_k$ if $k \notin \{i,i+1\}$.
This is a contradiction.
Hence we have $m_1 \ge m_2 \ge \cdots \ge m_r$.
If $m_1 \ge 4$, then 
$$
m_1 \le \left\lfloor \frac{m_1+1}{2} \right\rfloor
\left(
m_1-
\left\lfloor \frac{m_1+1}{2} \right\rfloor
\right).
$$
Hence 
$$
2  m_1 m_2  \cdots m_r
+
\sum_{j=1}^{r-1}
\prod_{k=1}^j m_k
<
2  m_0' m_1'  \cdots m_r'
+
\sum_{j=0}^{r-1}
\prod_{k=0}^j m_k',
$$
where $m_0' = \left\lfloor \frac{m_1+1}{2} \right\rfloor$,
$m_1' = m_1-m_0'$ and
$m_k' = m_k$ if $k \notin \{0, 1\}$. 
This is a contradiction.
Thus we have $m_1 \le 3$.
It is easy to see that $m_r \neq 1$.
Therefore 
$$
3 \ge m_1 \ge m_2 \ge \cdots \ge m_r \ge 2
.$$
Since
$2 + 2 + 2 + 2 = 3 + 3 +2$ and
$
2 + 2^2 + 2^3 + 2 \cdot 2^4 =46
<
48=
3 + 3^2 + 2 \cdot 3^2 \cdot 2 
$,
there are at most three $m_i$'s that are equal to $2$.
If $n = 3k+1$, then
$m_1 = \cdots = m_{r-2} =3$ and $m_{r-1} = m_r=2$.
If $n = 3k+2$, then
$m_1 = \cdots = m_{r-1} =3$ and $m_r=2$.
If $n = 3k \ge 6$, then there are two possibilities: 
\begin{equation}
\label{zureta}
m_1 = \cdots = m_{r-3} =3, \mbox{ and }
 m_{r-2} = m_{r-1} = m_r=2,
\end{equation}
\begin{equation}
m_1 = \cdots = m_r=3.
\end{equation}
Since 
$
2 + 2^2 + 2 \cdot 2^3 
= 22
>
21=
3 + 2 \cdot 3^2 
$, 
it follows that $m_1,\dots,m_r$ satisfies (\ref{zureta}).

Thus the maximum value is equal to 
$$
\left\{
\begin{array}{lcll}
2 \cdot  3^{k-2} \cdot 2^3
+
\sum_{j=1}^{k-2} 3^j + 3^{k-2}(2 + 2^2 )
 & = & \frac{47}{2} \cdot 3^{k-2} -\frac{3}{2} & \mbox{ if } n = 3k ,\\
\\
2 \cdot  3^{k-1} \cdot 2^2
+
\sum_{j=1}^{k-1} 3^j + 3^{k-1} \cdot 2
 & = & \frac{23}{2} \cdot 3^{k-1} -\frac{3}{2} &\mbox{ if }  n = 3k+1,\\
\\
2 \cdot  3^k \cdot 2
+
\sum_{j=1}^k 3^j 
 & = & \frac{11}{2} \cdot 3^k -\frac{3}{2}&\mbox{ if }  n = 3k+2.
\end{array}
\right.
$$
A poset that attains the maximum value is the ordinal sum
$A_r \oplus \cdots \oplus A_1$
of antichains $A_1,\dots,A_r$ such that $|A_i| = m_i$.
\end{proof}

Finally, we discuss when the number of facets of $\Oc^{(e)}_P$ and $\Cc^{(e)}_P$ are coincide.

\begin{Proposition}
\label{rare}
	Let $P$ be a finite poset on $[n]$.
	Then the following conditions are equivalent{\rm :}
	\begin{enumerate}
		\item[{\rm (i)}] $P$ is an antichain{\rm ;}
		\item[{\rm (ii)}] $\Oc^{(e)}_P$ and $\Cc^{(e)}_P$ are unimodularly equivalent{\rm ;}
		\item[{\rm (iii)}] $\Oc^{(e)}_P$ is centrally symmetric{\rm ;}
		\item[{\rm (iv)}] The number of the facets of $\Oc^{(e)}_P$ is equal to that of $\Cc^{(e)}_P$.
	\end{enumerate}
	\end{Proposition}

\begin{proof}
First, (ii) $\Rightarrow$ (iv) is trivial.

(ii) $\Rightarrow$ (iii): 
Note that $\Cc^{(e)}_P$ is always centrally symmetric,
and that the origin is the unique interior lattice point in each of $\Cc^{(e)}_P$ and $\Oc^{(e)}_P$.
	Hence if  $\Oc^{(e)}_P$ and $\Cc^{(e)}_P$ are unimodularly equivalent, then $\Oc^{(e)}_P$ is also centrally symmetric.

(iii) $\Rightarrow$ (i): 
	Assume that $\Oc^{(e)}_P$ is centrally symmetric.
	Then since $\eb_1+\cdots+\eb_n \in \Oc^{(e)}_P$, 
	one has $-\eb_1-\cdots-\eb_n \in \Oc^{(e)}_P$.
	By the definition of $\Oc^{(e)}_P$, this implies that each element of $P$ is a minimal element of $P$.
	Hence $P$ is an antichain.

(i) $\Rightarrow$ (ii): 
If $P$ is an antichain, then we have $\Oc^{(e)}_P = \Cc^{(e)}_P$.

(iv) $\Rightarrow$ (i):  
Suppose that the number of the facets of $\Oc^{(e)}_P$ is equal to that of $\Cc^{(e)}_P$.
By the argument in the proof of Corollary~\ref{numberoffacets}, 
each maximal chain of $P$ of length $\ell$ must satisfy
 $\ell + 2 = 2^{\ell+1}$, and hence $\ell = 0$.
Thus $P$ is an antichain.
\end{proof}

\end{document}